\documentclass[11pt]{amsart}

\author{Carlo Sanna}
\address{Universit\`a degli Studi di Torino\\Department of Mathematics\\Turin, Italy}
\email{carlo.sanna.dev@gmail.com}

\title[Distribution of integral values for the ratio of two linear recurrences]{Distribution of integral values for the ratio of two~linear~recurrences}

\usepackage{amsmath}
\usepackage{amssymb}
\usepackage{amsthm}
\usepackage{geometry}
\usepackage{color}
\usepackage{hyperref}
\usepackage{enumerate}
\hypersetup{colorlinks=true}

\newtheorem{thm}{Theorem}[section]
\newtheorem{cor}{Corollary}[section]
\newtheorem{lem}[thm]{Lemma}
\theoremstyle{remark}

\def\Li{\operatorname{Li}}

\begin{document}

\begin{abstract}
Let $F$ and $G$ be linear recurrences over a number field $\mathbb{K}$, and let $\mathfrak{R}$ be a finitely generated subring of $\mathbb{K}$.
Furthermore, let $\mathcal{N}$ be the set of positive integers $n$ such that $G(n) \neq 0$ and $F(n) / G(n) \in \mathfrak{R}$.
Under mild hypothesis, Corvaja and Zannier proved that $\mathcal{N}$ has zero asymptotic density.
We prove that $\#(\mathcal{N} \cap [1, x]) \ll x \cdot (\log\log x / \log x)^h$ for all $x \geq 3$, where $h$ is a positive integer that can be computed in terms of $F$ and $G$.
Assuming the Hardy--Littlewood $k$-tuple conjecture, our result is optimal except for the term $\log \log x$. 
\end{abstract}

\maketitle

\section{Introduction}

A sequence of complex numbers $F(n)_{n \in \mathbb{N}}$ is called a \emph{linear recurrence} if there exist some $c_0, \ldots, c_{k-1} \in \mathbb{C}$ ($k \geq 1$), with $c_0 \neq 0$, such that
\begin{equation*}
F(n + k) = \sum_{j = 0}^{k - 1} c_j F(n + j) ,
\end{equation*}
for all $n \in \mathbb{N}$.
In turn, this is equivalent to an (unique) expression
\begin{equation*}
F(n) = \sum_{i = 1}^r f_i(n) \, \alpha_i^n ,
\end{equation*}
for all $n \in \mathbb{N}$, where $f_1, \ldots, f_r \in \mathbb{C}[X]$ are nonzero polynomials and $\alpha_1, \ldots, \alpha_r \in \mathbb{C}^*$ are all the distinct roots of the polynomial
\begin{equation*}
X^k - c_{k - 1} X^{k - 1} - \cdots - c_1 X - c_0 .
\end{equation*}
Classically, $\alpha_1, \ldots, \alpha_r$ and $k$ are called the \emph{roots} and the \emph{order} of $F$, respectively.
Furthermore, $F$ is said to be \emph{nondegenerate} if none the ratios $\alpha_i / \alpha_j$ ($i \neq j$) is a root of unity, and $F$ is said to be \emph{simple} if all the $f_1, \ldots, f_r$ are constant.
We refer the reader to~\cite[Ch.~1--8]{MR1990179} for the general theory of linear recurrences.

Hereafter, let $F$ and $G$ be linear recurrences and let $\mathfrak{R}$ be a finitely generated subring of~$\mathbb{C}$.
Assume also that the roots of $F$ and $G$ together generate a multiplicative torsion-free group.
This ``torsion-free'' hypothesis is not a loss of generality.
Indeed, if the group generated by the roots of $F$ and $G$ has torsion order $q$, then for each $r = 0,1,\ldots,q-1$ the roots of the linear recurrences $F_r(n) = F(qn + r)$ and $G_r(n) = G(qn + r)$ generate a torsion-free group.
Therefore, all the results in the following can be extended just by partitioning $\mathbb{N}$ into the arithmetic progressions of modulo $q$ and by studying each pair of linear recurrences $F_r, G_r$ separately.
Finally, define the following set of natural numbers
\begin{equation*}
\mathcal{N} := \left\{n \in \mathbb{N} : G(n) \neq 0, \; F(n) / G(n) \in \mathfrak{R} \right\} .
\end{equation*}
Regarding the condition $G(n) \neq 0$, note that, by the ``torsion-free'' hypothesis, $G(n)$ is nondegenerate and hence the Skolem--Mahler--Lech Theorem~\cite[Theorem~2.1]{MR1990179} implies that $G(n) = 0$ only for finitely many $n \in \mathbb{N}$.
In the sequel, we shall tacitly disregard such integers.

Divisibility properties of linear recurrences have been studied by several authors.
A classical result, conjectured by Pisot and proved by van der Poorten, is the Hadamard-quotient Theorem, which states that if $\mathcal{N}$ contains all sufficiently large integers, then $F / G$ is itself a linear recurrence~\cite{MR990517, MR929097}.

Corvaja and Zannier~\cite[Theorem~2]{MR1918678} gave the following wide extension of the Hadamard-quotient Theorem (see also~\cite{MR1692189} for a previous weaker result by the same authors).

\begin{thm}\label{thm:infinite}
If $\mathcal{N}$ is infinite, then there exists a nonzero polynomial $P \in \mathbb{C}[X]$ such that both the sequences $n \mapsto P(n) F(n) / G(n)$ and $n \mapsto G(n) / P(n)$ are linear recurrences.
\end{thm}

The proof of Theorem~\ref{thm:infinite} makes use of the Schmidt's Subspace Theorem.
We refer the reader to~\cite{MR2487729} for a survey on several applications of the Schmidt's Subspace Theorem in Number~Theory.

Let $\mathbb{K}$ be a number field.
For the sake of simplicity, from now on we shall assume that $\mathfrak{R} \subseteq \mathbb{K}$ and that $F$ and $G$ have coefficients and values in $\mathbb{K}$.
Corvaja and Zannier~\cite[Corollary~2]{MR1918678} proved also the following theorem about the set $\mathcal{N}$.
\begin{thm}\label{thm:zerodensity}
If $F / G$ is not a linear recurrence, then $\mathcal{N}$ has zero asymptotic density.
\end{thm}

We recall that a set of natural numbers $\mathcal{S}$ has zero asymptotic density if $\#\mathcal{S}(x) / x \to 0$, as $x \to +\infty$, where we define $\mathcal{S}(x) := \mathcal{S} \cap [1, x]$ for all $x \geq 1$.

Corvaja and Zannier also suggested~\cite[Remark p.~450]{MR1918678} that their proof of Theorem~\ref{thm:zerodensity} could be adapted to show that if $F/G$ is not a linear recurrence then
\begin{equation}\label{equ:czubound}
\#\mathcal{N}(x) \ll \frac{x}{(\log x)^{\delta}} ,
\end{equation}
for any $\delta < 1$ and for all sufficiently large $x > 1$, where the implied constant depends on $\mathbb{K}$.

In our main result we obtain a more precise upper bound than (\ref{equ:czubound}).
Before state it, we mention some special cases of the problem of bounding $\#\mathcal{N}(x)$ that have already been studied.

Alba Gonz{\'a}lez, Luca, Pomerance, and Shparlinski~\cite[Theorem~1.1]{MR2928495} proved the following:
\begin{thm}
If $F$ is a simple nondegenerate linear recurrence over the integers, $r \geq 2$, $G(n) = n$, and $\mathcal{R} = \mathbb{Z}$, then
\begin{equation*}
\#\mathcal{N}(x) \ll \frac{x}{\log x} ,
\end{equation*}
for all sufficiently large $x > 1$, where the implied constant depends only on $r$.
\end{thm}

For $G(n) = n$ and $\mathcal{R} = \mathbb{Z}$, a still better upper bound can be given if $F$ is a Lucas sequence, that is, $F(0) = 0$, $F(1) = 1$, and $F(n + 2) = a F(n + 1) + b F(n)$, for all $n \in \mathbb{N}$ and some fixed integers $a$ and $b$.
In such a case the arithmetic properties of $\mathcal{N}$ were first investigated by Andr{\'e}-Jeannin~\cite{MR1131414} and Somer~\cite{MR1271392, MR1393479}.
Luca and Tron~\cite{MR3409327} studied the case in which $F$ is the sequence of Fibonacci numbers ($a = b = 1$) and Sanna~\cite{San16}, using some results on the $p$-adic valuation of Lucas sequences \cite{San16bis}, generalized Luca and Tron's result to the following upper bound. 

\begin{thm}
If $F$ is a nondegenerate Lucas sequences, $G(n) = n$, and $\mathcal{R} = \mathbb{Z}$, then
\begin{equation*}
\#\mathcal{N}(x) \leq x^{1 - \left(\frac1{2} + o(1)\right)\frac{\log\log\log x}{\log\log x}} ,
\end{equation*}
as $x \to +\infty$, where the $o(1)$ depends on $F$.
\end{thm}

Now we state the main result of this paper.

\begin{thm}\label{thm:main}
If $F / G$ is not a linear recurrence, then
\begin{equation*}
\#\mathcal{N}(x) \ll_{F,G} x \cdot \left(\frac{\log \log x}{\log x}\right)^h ,
\end{equation*}
for all $x \geq 3$, where $h$ is a positive integer depending on $F$ and $G$.
\end{thm}

Both the positive integer $h$ and the implied constant in the bound of Theorem~\ref{thm:main} are effectively computable, we give the details in the last section.
In particular, we have the following corollary.

\begin{cor}\label{cor:integers}
If $F/G$ is not a linear recurrence, $G \in \mathbb{Z}[X]$, and $\gcd(G, f_1, \ldots, f_r) = 1$, then $h$ can be taken as the number of irreducible factors of $G$ in $\mathbb{Z}[X]$ (counted without multiplicity).
\end{cor}

Except for the term $\log \log x$, Corollary~\ref{cor:integers} should be optimal.
Indeed, pick a positive integer $h$ and an \emph{admissible} $h$-tuple $\mathbf{h} = (n_1, \ldots, n_h)$, that is, $n_1 < \cdots < n_h$ are positive integers such that for each prime number $p$ there exists a residue class modulo $p$ which does not intersect $\{n_1, \ldots, n_h\}$.
Assuming Hardy--Littlewood $h$-tuple conjecture~\cite[p.~61]{MR1555183}, we have that the number $T_{\mathbf{h}}(x)$ of positive integers $n \leq x$ such that $n + n_1, \ldots, n + n_h$ are all prime numbers satisfies
\begin{equation*}
T_{\mathbf{h}}(x) \sim C_{\mathbf{h}}\cdot \frac{x}{(\log x)^h} ,
\end{equation*}
as $x \to +\infty$, where $C_{\mathbf{h}} > 0$ depends on $\mathbf{h}$.
Therefore, taking $F(n) = (2^{n+n_1} - 2) \cdots (2^{n + n_h} - 2)$ and $G(n) = (n + n_1) \cdots (n + n_h)$, we obtain
\begin{equation*}
\#\mathcal{N}(x) \geq T_{\mathbf{h}}(x) \gg \frac{x}{(\log x)^h} ,
\end{equation*}
for all sufficiently large $x > 1$.

\subsection*{Notation}

Hereafter, the letter $p$ always denotes a prime number.
We employ the Landau--Bachmann ``Big Oh'' and ``little oh'' notations $O$ and $o$, as well as the associated Vinogradov symbols $\ll$ and $\gg$, with their usual meanings.
If $A \ll B$ and $A \gg B$, we write $A \asymp B$.
Any dependence of implied constants is explicitly stated or indicated with subscripts.

\section{Preliminaries}

First, we need a quantitative form of a result due to Kronecker~\cite{Kro80} (see also~\cite[p.~32]{MR1395088}), which states that the average number of zeros modulo $p$ of a nonconstant polynomial $f \in \mathbb{Z}[X]$ is equal to the number of irreducible factors of $f$ in $\mathbb{Z}[X]$.

\begin{lem}\label{lem:zerosmodp}
Given a nonconstant polynomial $f \in \mathbb{Z}[X]$, for each prime number $p$ let $\eta_f(p)$ be the number of zeros of $f$ modulo $p$.
Then
\begin{equation*}
\sum_{p \leq x} \eta_f(p) \cdot \frac{\log p}{p} = h \log x + O_f(1) ,
\end{equation*}
for all $x \geq 1$, where $h$ is the number of irreducible factors of $f$ in $\mathbb{Z}[X]$.
\end{lem}
\begin{proof}
It is enough to prove the claim for irreducible $f$.
Let $\mathbb{L}$ be the splitting field of $f$ over $\mathbb{Q}$ and let $\mathcal{G} := \operatorname{Gal}(\mathbb{L} / \mathbb{Q})$.
For any coniugacy class $C$ of $\mathcal{G}$, let $\pi_C(x)$ be the number of primes $p \leq x$ which do not ramified in $\mathbb{L}$ and such that their Frobenius substitutions $\sigma_p$ belong to $C$. 
A quantitative version of the Chebotarev's density theorem~\cite[Theorem~3.4]{MR2920749} states that
\begin{equation*}
\pi_C(x) = \frac{\#C}{\#\mathcal{G}} \cdot \Li(x) + O_\mathbb{L}\!\left(\frac{x}{\exp(C \sqrt{\log x})}\right) ,
\end{equation*}
for $x \to +\infty$, where $\Li(x)$ is the logarithmic integral function and $C > 0$ is a constant depending on $\mathbb{L}$.
If the elements of $C$ have cycle pattern $d_1, \ldots, d_s$, when regarded as permutations of the roots of $f$, then $\pi_C(x)$ is the number of primes $p \leq x$ not dividing the discriminant of $f$ and such that the irreducible factors of $f$ modulo $p$ have degrees $d_1,\ldots,d_s$.

Furthermore, $\mathcal{G}$ acts transitively on the roots of $f$, since $f$ is irreducible, hence
\begin{equation*}
\sum_{g \in \mathcal{G}} \#X^g = \#\mathcal{G} ,
\end{equation*}
by Burnside's lemma, where $X^g$ is the set of roots of $f$ which are fixed by $g$.
Therefore,
\begin{equation*}
\sum_{p \leq x} \eta_f(p) = \Li(x) + O_\mathbb{L}\!\left(\frac{x}{\exp(C \sqrt{\log x})}\right) ,
\end{equation*}
and the desired result follows by partial summation.
\end{proof}

The following lemma~\cite[Lemma~A.2]{MR1918678} regards the minimum of the multiplicative orders of some fixed algebraic numbers modulo a prime ideal.

\begin{lem}\label{lem:ponequarter}
Let $\beta_1, \ldots, \beta_s \in \mathbb{K}$ such that none of them is zero or a root of unity.
Then, for all $x \geq 1$, the number of prime numbers $p \leq x$ such that some $\beta_i$ has order less than $p^{1/4}$ modulo some prime ideal of $\mathcal{O}_\mathbb{K}$ lying above $p$ is $O(x^{1/2})$, where the implied constant depends only on $\beta_1, \ldots, \beta_s$.
\end{lem}

{%\color{blue}
Given a multiplicative function $g$, let $\Lambda_g$ be its \emph{associated von Mangoldt function}, that is, the unique arithmetic function satisfying
\begin{equation*}
\sum_{d \,\mid\, n} g(n / d) \Lambda_g(d) = g(n) \log n ,
\end{equation*}
for all positive integers $n$ (see \cite[p.~17]{MR2061214}).
It is easy to prove that $\Lambda_g$ is supported on prime powers.

\begin{thm}\label{thm:wir}
Let $g$ be a multiplicative arithmetic function such that
\begin{equation}\label{equ:wir1}
\sum_{n \leq x} \Lambda_g(n) = h \log x + O(L)
\end{equation}
and
\begin{equation}\label{equ:wir2}
\sum_{n \leq x} |g(n)| \ll (\log x)^h ,
\end{equation}
for all $x \geq 2$, where $h, L > 0$ are some constants.
Then
\begin{equation*}
\sum_{n \leq x} g(n) = (\log x)^h \cdot \left(c_g + O_h\!\left(\frac{L}{\log x}\right)\right) ,
\end{equation*}
where
\begin{equation*}
c_g := \frac1{\Gamma(h + 1)} \prod_p \left(1 + g(p) + g(p^2) + \cdots \right)\left(1 - \frac1{p}\right)^h
\end{equation*}
and $\Gamma$ is the Euler's Gamma function.
\end{thm}
\begin{proof}
The proof proceeds exactly as the proof of \cite[Theorem~1.1]{MR2061214}, but using the error term $O(L)$ instead of $O(1)$.
\end{proof}

Now we state a technical lemma about the cardinality of a sieved set of integers.

\begin{lem}\label{lem:sieve}
For each prime number $p$, let $\Omega_p \subsetneq \{0, 1, \ldots, p - 1\}$ be a set of residues modulo $p$.
Suppose that there exist constants $c_1, c_2, h > 0$ such that $\#\Omega_p \leq c_1$ for each prime number $p$ and
\begin{equation}\label{equ:Omegaplogp}
\sum_{p \leq x} \#\Omega_p \cdot \frac{\log p}{p} = h \log x + O(c_2) ,
\end{equation}
for all $x > 1$.
Then we have
\begin{equation*}
\#\left\{n \leq x : (n \bmod p) \notin \Omega_p, \; \forall p \in {]y,z]}\right\} \ll_{c_1, c_2, h, \delta_1, \delta_2} x \cdot \left(\frac{\log y}{\log x}\right)^h ,
\end{equation*}
for all $\delta_1, \delta_2 > 0$, $x > 1$, $2 \leq y \leq (\log x)^{\delta_1}$, and $z \geq x^{\delta_2}$.
\end{lem}
\begin{proof}
All the constants in this proof, included the implied ones, may depend on $c_1$, $c_2$, $h$, $\delta_1$, $\delta_2$.
Clearly, we can assume $\delta_2 \leq 1/2$.
By the large sieve inequality~\cite[Theorem~7.14]{MR2061214}, we have
\begin{equation}\label{equ:sieve1}
\#\left\{n \leq x : (n \bmod p) \notin \Omega_p, \; \forall p \in {]y,z]}\right\} \ll x \cdot \left(\sum_{m \leq w} g_y(m) \right)^{-1} ,
\end{equation}
where $w := x^{\delta_2}$ and $g_y$ is the multiplicative arithmetic function supported on squarefree numbers with all prime factors $> y$ and such that
\begin{equation*}
g_y(p) = \frac{\#\Omega_p}{p - \#\Omega_p} ,
\end{equation*}
for any prime number $p > y$.

For sufficiently large $x$, we have $y \leq w$, and it follows from (\ref{equ:Omegaplogp}) and $\#\Omega_p \leq c_1$ that
\begin{equation*}
\sum_{p \leq w} g_y(p) \log p = h \log w + O(\log y) ,
\end{equation*}
which in turn implies that
\begin{equation*}
\sum_{n \leq w} \Lambda_{g_y}(n) = h \log w + O(\log y) ,
\end{equation*}
since $\Lambda_{g_y}$ is supported on prime powers $p^s$, with $p > y$, and $\Lambda_{g_y}(p^s) = -(-g_y(p))^s \log p$.

Furthermore, again from (\ref{equ:Omegaplogp}) and $\#\Omega_p \leq c_1$, we have
\begin{equation}\label{equ:1minusOmega}
\prod_{p \leq t} \left(1 - \frac{\#\Omega_p}{p}\right)^{-1} \asymp (\log t)^h ,
\end{equation}
for all $t \geq 2$, so that
\begin{equation*}
\sum_{n \leq w} |g_y(n)| \leq \prod_{p \leq w} (1 + g_y(p)) \leq \prod_{p \leq w} \left(1 - \frac{\#\Omega_p}{p}\right)^{-1} \ll (\log w)^h .
\end{equation*}

At this point, we have proved that (\ref{equ:wir1}) and (\ref{equ:wir2}) hold with $g = g_y$ and $L = \log y$.
Therefore, by Theorem~\ref{thm:wir} we have
\begin{equation}\label{equ:s1}
\sum_{n \leq w} g(n) = (\log w)^h \cdot \left(c_{g_y} + O\!\left(\frac{\log y}{\log w}\right)\right) ,
\end{equation}
where
\begin{equation*}
c_{g_y} = \frac1{\Gamma(h + 1)} \prod_p (1 + g_y(p))\left(1 - \frac1{p}\right)^h .
\end{equation*}
Now using (\ref{equ:1minusOmega}) we obtain
\begin{equation}\label{equ:cgy}
c_{g_y} = \frac1{\Gamma(h + 1)} \prod_p \left(1 - \frac{\#\Omega_p}{p}\right)^{-1} \left(1 - \frac1{p}\right)^h \prod_{p \leq y} \left(1 - \frac{\#\Omega_p}{p}\right) \gg \frac1{(\log y)^h} .
\end{equation}
Hence, recalling that $y \leq (\log x)^{\delta_1}$ and $w = x^{\delta_2}$, by (\ref{equ:s1}) and (\ref{equ:cgy}) we find that
\begin{equation}\label{equ:sieve2}
\sum_{n \leq w} g(n) \gg \left(\frac{\log w}{\log x}\right)^h \gg \left(\frac{\log x}{\log y}\right)^h.
\end{equation}
Putting together (\ref{equ:sieve1}) and (\ref{equ:sieve2}), the desired result follows.
\end{proof}
}
Finally, we need a lemma about the number of zeros of a simple linear recurrence in a finite field of $q$ elements $\mathbb{F}_q$.

\begin{lem}\label{lem:finitefield}
Let $c_1, \ldots, c_r, a_1, \ldots, a_r \in \mathbb{F}_q^*$, and let $N$ be the minimum of the orders of the $a_i / a_j$ ($i \neq j$) in $\mathbb{F}_q^*$.
(If $r = 1$ then pick an arbitrary positive integer $N$.)
Then the number of integers $m \in [0, q - 2]$ such that
\begin{equation*}
\sum_{i=1}^r c_i a_i^m = 0
\end{equation*}
is at most $4(q - 1)N^{-1/2^{r - 2}}$.
\end{lem}
\begin{proof}
In~\cite[Proposition~A.1]{MR1918678} the claim is stated and proved for prime $q$, but the same proof works also for not necessarily prime $q$.
\end{proof}

Note that results stronger than Lemma~\ref{lem:finitefield} can be obtained using bounds for the number of zeros of sparse polynomials in finite fields (see, e.g., \cite{MR1815369, MR3540955}).

\section{Proof of Theorem~\ref{thm:main}}

The first part of the proof proceeds similarly to the proof of Theorem~\ref{thm:zerodensity}.
If $\mathcal{N}$ is finite, then the claim is trivial, hence we suppose that $\mathcal{N}$ is infinite.
Then, by Theorem~\ref{thm:infinite} it follows that $F / G = H / P$, for some linear recurrence $H$ and some polynomial $P$.
As a consequence, without loss of generality, we shall assume that $G$ is a polynomial.

Let $S$ be a finite set of absolute values of $\mathbb{K}$ containing all the archimedean ones.
Write $\mathcal{O}_S$ for the ring of $S$-integers of $\mathbb{K}$, that is, the set of all $\alpha \in \mathbb{K}$ such that $|\alpha|_v \leq 1$ for all $v \notin S$.
Enlarging $\mathbb{K}$ and $S$ we may assume that $\alpha_1, \ldots, \alpha_r$ are $S$-units, $f_1, \ldots, f_r, G \in \mathcal{O}_S[X]$, and $\mathfrak{R} \subseteq \mathcal{O}_S$.

Since $F / G$ is not a linear recurrence, it follows that $G$ does not divide all the $f_1, \ldots, f_r$.
Moreover, factoring out the greatest common divisor $(G, f_1, \ldots, f_r)$ we can even assume that $(G, f_1 , \ldots, f_r) = 1$ and that $G$ is nonconstant.
In particular, $(G(n), f_1(n), \ldots , f_r(n))$ is bounded and, enlarging $S$, we may assume that it is an $S$-unit for all $n \in \mathbb{N}$.

Let $N_\mathbb{K}(\alpha)$ denotes the norm of $\alpha \in \mathbb{K}$ over $\mathbb{Q}$.
It is easy to prove that there exist a positive integer $g$ and a nonconstant polynomial $\widetilde{G} \in \mathbb{Z}[X]$ such that $N_\mathbb{K}(G(n)) = \widetilde{G}(n) / g$ for all $n \in \mathbb{N}$.
Let $h$ be the number of irreducible factors of $\widetilde{G}$ in $\mathbb{Z}[X]$.
Again by enlarging $S$, we may assume that $g$ is an $S$-unit.

Let $\mathcal{P}$ be the set of all prime numbers $p$ which do not make $\widetilde{G}$ vanish identically modulo $p$, such that $p \mathcal{O}_\mathbb{K}$ has no prime ideal factor $\pi_v$ with $v \in S$, and such that the minimum order of the $\alpha_i / \alpha_j$ ($i \neq j$) modulo any prime ideal above $p$ is at least $p^{1/4}$.
Furthermore, let us define 
\begin{equation*}
\Omega_p := \left\{\ell \in \{0, \ldots, p - 1\} : \widetilde{G}(\ell) \equiv 0 \pmod p\right\} ,
\end{equation*}
for any $p \in \mathcal{P}$, and $\Omega_p := \varnothing$ for any prime number $p \notin \mathcal{P}$.

Let $x \geq 3$, $y := (\log x)^{2^r h}$, and $z := x^{1/(d + 1)}$, where $d := [\mathbb{K} : \mathbb{Q}]$.
We split $\mathcal{N}(x)$ into two subsets:
\begin{align*}
\mathcal{N}_1 &:= \left\{n \in \mathcal{N}(x) : (n \bmod p) \notin \Omega_p, \; \forall p \in {]y, z]}\right\} , \\
\mathcal{N}_2 &:= \mathcal{N} \setminus \mathcal{N}_1 .
\end{align*}

First, we give an upper bound for $\#\mathcal{N}_1$.
Hereafter, all the implied constants may depend on $F$ and $G$.
Clearly, $\#\Omega_p \subsetneq \{0,1,\ldots,p-1\}$ and $\#\Omega_p \leq \deg(\widetilde{G})$ for all prime number $p$, while from Lemma~\ref{lem:zerosmodp} and Lemma~\ref{lem:ponequarter} it follows that
\begin{equation*}
\sum_{p \leq x} \#\Omega_p \cdot \frac{\log p}{p} = h \log x + O(1) .
\end{equation*}
Therefore, applying Lemma~\ref{lem:sieve}, we obtain
\begin{equation*}
\#\mathcal{N}_1 \ll x \cdot \left(\frac{\log y}{\log x}\right)^h \ll \left(\frac{\log \log x}{\log x}\right)^h.
\end{equation*}

Now we give an upper bound for $\#\mathcal{N}_2$.
If $n \in \mathcal{N}_2$ then there exist $p \in \mathcal{P} \cap {]y, z]}$ and $\ell \in \Omega_p$ such that $n \equiv \ell \pmod p$.
In particular, $p$ divides $N_\mathbb{K}(G(\ell))$ in $\mathcal{O}_S$ and, since $p\mathcal{O}_\mathbb{K}$ has no prime ideal factor $\pi_v$ with $v \in S$, it follows that there exists some prime ideal $\pi$ of $\mathcal{O}_S$ lying above $p$ and dividing $G(\ell)$.
Let $\mathbb{F}_q := \mathcal{O}_S / \pi$, so that $q$ is a power of $p$.
Write $n = \ell + mp$, for some integer $m \geq 0$. 
Since $\pi$ divides $G(n)$ and $F(n) / G(n) \in \mathcal{O}_S$, we have that $F(n)$ is divisible by $\pi$ too.
As a consequence, we obtain that
\begin{equation}\label{equ:Fmodpi}
\sum_{i = 1}^r f_i(\ell)\, \alpha_i^\ell\, \big(\alpha_i^p\big)^m \equiv \sum_{i = 1}^r f_i(n)\, \alpha_i^n \equiv F(n) \equiv 0 \pmod \pi.
\end{equation}
Note that $f_1(\ell), \ldots, f_r(\ell)$ cannot be all equal to zero modulo $\pi$, since $\pi$ divides $G(\ell)$ and $(G(\ell), f_1(\ell), \ldots, f_r(\ell))$ is an $S$-unit.
Note also that the minimum order $N$ of the $\alpha_i^p / \alpha_j^p$ ($i \neq j$) modulo $\pi$ is equal to the minimum order of the $\alpha_i / \alpha_j$ ($i \neq j$) modulo $\pi$, since $(p, q - 1) = 1$.
In particular, $N \geq p^{1/4}$, in light of the definition of $\mathcal{P}$. 

Therefore, we can apply Lemma~\ref{lem:finitefield} to the congruence (\ref{equ:Fmodpi}), getting that the number of possible values of $m$ modulo $q - 1$ is at most $4(q - 1) p^{-1/2^r}$.
Consequently, the number of possible values of $n \leq x$ is at most 
\begin{equation*}
4(q - 1) p^{-1/2^r} \left(\frac{x}{p(q - 1)} + 1\right) \ll \frac{x}{p^{1 + 1/2^r}} ,
\end{equation*}
since $p(q - 1) < p^{d + 1} \leq z^{d + 1} \leq x$.
Hence, we have
\begin{equation*}
\#\mathcal{N}_2 \ll \sum_{p \in \mathcal{P} \,\cap\, {]y, z]}} \frac{x}{p^{1 + 1/2^r}} \ll \int_y^{+\infty} \frac{\mathrm{d}t}{t^{1 + 1/2^r}} \ll \frac{x}{y^{1/2^r}} = \frac{x}{(\log x)^h} .
\end{equation*}
In conclusion,
\begin{equation*}
\#\mathcal{N}(x) = \#\mathcal{N}_1 + \#\mathcal{N}_2 \ll x \cdot \left(\frac{\log \log x}{\log x}\right)^h
\end{equation*}
as claimed.

\section{Concluding remarks}

Let us briefly explain the computation of $h$.
First, we have an effective procedure to test if there exists a nonzero polynomial $P \in \mathbb{C}[X]$ such that the sequences $n \mapsto P(n) F(n) / G(n)$ and $n \mapsto G(n) / P(n)$ are linear recurrences, and in such a case $P$ can be determined (see \cite[p.~435, Remark~1]{MR1918678}).

On the one hand, if $P$ does not exist, then Theorem~\ref{thm:infinite} implies that $\mathcal{N}$ is finite, hence $h$ can be any positive integer.
Moreover, using any effective version of the Skolem--Mahler--Lech Theorem at the end of the proof of \cite[Proposition~2.1]{MR1918678}, it is possible to bound $\#\mathcal{N}$.
Therefore, the implied constant in Theorem~\ref{thm:main} is effectively computable.

On the other hand, if $P$ exists, then we can write the linear recurrences $H = PF / G$ as
\begin{equation*}
H(n)  = \sum_{i = 1}^s h_i(n) \, \beta_i^n ,
\end{equation*}
for some $\beta_1, \ldots, \beta_s \in \mathbb{C}^*$ and $h_1, \ldots, h_s \in \mathbb{C}[X]$.
Setting $Q := P / (P, h_1, \ldots, h_s)$, we have that $\widetilde{Q}(n) = N_\mathbb{K}(Q(n))$ is a polynomial in $\mathbb{Q}[X]$ and $h$ can be taken as the number of irreducible factors of $\widetilde{Q}$.
Furthermore, the implied constant in Theorem~\ref{thm:main} is effectively computable, since all the implied constants of the results used in the proof of Theorem~\ref{thm:main} are effectively computable. 

\subsection*{Acknowledgements}
The author thanks Umberto Zannier for a fruitful conversation on Theorem~\ref{thm:zerodensity}, and also the anonymous referee for useful comments which improved the quality of the paper.

\bibliographystyle{amsplain}

\begin{thebibliography}{10}

\bibitem{MR2928495}
J.~J. Alba~Gonz{\'a}lez, F.~Luca, C.~Pomerance, and I.~E. Shparlinski, \emph{On
  numbers {$n$} dividing the {$n$}th term of a linear recurrence}, Proc. Edinb.
  Math. Soc. (2) \textbf{55} (2012), no.~2, 271--289.

\bibitem{MR1131414}
R.~Andr{\'e}-Jeannin, \emph{Divisibility of generalized {F}ibonacci and {L}ucas
  numbers by their subscripts}, Fibonacci Quart. \textbf{29} (1991), no.~4,
  364--366.

\bibitem{MR2487729}
Y.~F. Bilu, \emph{The many faces of the subspace theorem [after {A}damczewski,
  {B}ugeaud, {C}orvaja, {Z}annier{$\ldots$}]}, Ast\'erisque (2008), no.~317,
  Exp. No. 967, vii, 1--38, S{\'e}minaire Bourbaki. Vol. 2006/2007.

\bibitem{MR1815369}
R.~Canetti, J.~Friedlander, S.~Konyagin, M.~Larsen, D.~Lieman, and
  I.~Shparlinski, \emph{On the statistical properties of {D}iffie-{H}ellman
  distributions}, Israel J. Math. \textbf{120} (2000), no.~part A, 23--46.

\bibitem{MR1692189}
P.~Corvaja and U.~Zannier, \emph{Diophantine equations with power sums and
  universal {H}ilbert sets}, Indag. Math. (N.S.) \textbf{9} (1998), no.~3,
  317--332.

\bibitem{MR1918678}
P.~Corvaja and U.~Zannier, \emph{Finiteness of integral values for the ratio of
  two linear recurrences}, Invent. Math. \textbf{149} (2002), no.~2, 431--451.

\bibitem{MR1990179}
G.~Everest, A.~van~der Poorten, I.~Shparlinski, and T.~Ward, \emph{Recurrence
  sequences}, Mathematical Surveys and Monographs, vol. 104, American
  Mathematical Society, Providence, RI, 2003.

\bibitem{MR1555183}
G.~H. Hardy and J.~E. Littlewood, \emph{Some problems of `{P}artitio
  numerorum'; {III}: {O}n the expression of a number as a sum of primes}, Acta
  Math. \textbf{44} (1923), no.~1, 1--70.

\bibitem{MR2061214}
H.~Iwaniec and E.~Kowalski, \emph{Analytic number theory}, American
  Mathematical Society Colloquium Publications, vol.~53, American Mathematical
  Society, Providence, RI, 2004.

\bibitem{MR3540955}
Z.~Kelley, \emph{Roots of sparse polynomials over a finite field}, LMS J.
  Comput. Math. \textbf{19} (2016), no.~suppl. A, 196--204.

\bibitem{Kro80}
L.~Kronecker, \emph{{\"U}ber die {I}rreductibilit{\"a}t von {G}leichungen},
  {M}onatsberichte {K}{\"o}nigl. {P}reu{\ss}isch. {A}kad. {W}issenschaft.
  {B}erlin (1880), 155--162.

\bibitem{MR3409327}
F.~Luca and E.~Tron, \emph{The distribution of self-{F}ibonacci divisors},
  Advances in the theory of numbers, Fields Inst. Commun., vol.~77,
  pp.~149--158.

\bibitem{MR990517}
R.~Rumely, \emph{Notes on van der {P}oorten's proof of the {H}adamard quotient
  theorem. {I}, {II}}, S\'eminaire de {T}h\'eorie des {N}ombres, {P}aris
  1986--87, Progr. Math., vol.~75, Birkh\"auser Boston, Boston, MA, 1988,
  pp.~349--382, 383--409.

\bibitem{San16bis}
C.~Sanna, \emph{The {$p$}-adic valuation of {L}ucas sequences}, Fibonacci
  Quart. \textbf{54} (2016), no.~2, 118--124.

\bibitem{San16}
C.~Sanna, \emph{On numbers {$n$} dividing the {$n$}th term of a {L}ucas
  sequence}, Int. J. Number Theory \textbf{13} (2017), no.~3, 725--734.

\bibitem{MR2920749}
J.-P. Serre, \emph{Lectures on {$N_X (p)$}}, Chapman \& Hall/CRC Research Notes
  in Mathematics, vol.~11, CRC Press, Boca Raton, FL, 2012.

\bibitem{MR1271392}
L.~Somer, \emph{Divisibility of terms in {L}ucas sequences by their
  subscripts}, Applications of {F}ibonacci numbers, {V}ol. 5 ({S}t. {A}ndrews,
  1992), Kluwer Acad. Publ., Dordrecht, 1993, pp.~515--525.

\bibitem{MR1393479}
L.~Somer, \emph{Divisibility of terms in {L}ucas sequences of the second kind
  by their subscripts}, Applications of {F}ibonacci numbers, {V}ol.\ 6
  ({P}ullman, {WA}, 1994), Kluwer Acad. Publ., Dordrecht, 1996, pp.~473--486.

\bibitem{MR1395088}
P.~Stevenhagen and H.~W. Lenstra, Jr., \emph{Chebotar\"ev and his density
  theorem}, Math. Intelligencer \textbf{18} (1996), no.~2, 26--37.

\bibitem{MR929097}
A.~J. van~der Poorten, \emph{Solution de la conjecture de {P}isot sur le
  quotient de {H}adamard de deux fractions rationnelles}, C. R. Acad. Sci.
  Paris S\'er. I Math. \textbf{306} (1988), no.~3, 97--102.

\end{thebibliography}

\end{document}